\documentclass[a4paper, 12pt]{amsart}
\usepackage{graphicx}
\usepackage[english]{babel}
\usepackage{amsfonts}
\usepackage{amsmath, amscd, amssymb}
\usepackage{xy}
\input xy
\xyoption{all}
\usepackage{latexsym}
\def \R{\mathbb{R}}
\def \H{\mathbb{H}}

\newtheorem{theorem}{Theorem}[section]

\newtheorem{lemma}[theorem]{Lemma}
\newtheorem{proposition}[theorem]{Proposition}

\theoremstyle{remark}
\newtheorem{remark}[theorem]{Remark}

\begin{document}
\title[Approximations of functions by curvatures]
{Approximations of periodic functions to $\R^n$\\ by curvatures of closed 
curves}
\author{J.~Mostovoy}
\address{CINVESTAV, Col. San Pedro Zacatenco,
M\'exico, D.F. CP 07360}
\email{jacob.mostovoy@gmail.com}
\author{R.~Sadykov}
\address{CINVESTAV, Col. San Pedro Zacatenco,
M\'exico, D.F. CP 07360}
\email{rstsdk@gmail.com}


\date{}

\begin{abstract}
We show that for any $n$ real periodic functions $f_1,\dots, f_n$ with the same period, such that $f_i>0$ for $i<n$, and a real number $\varepsilon>0$, there is a closed curve in $\R^{n+1}$ with curvatures $\kappa_1, ..., \kappa_n$ such that $\left| \kappa_i(t)-f_i(t)\right|<\varepsilon$ for all $i$ and $t$.  This neither holds for closed curves in the hyperbolic  space $\H^{n+1}$, nor for parametric families of closed curves in $\R^{n+1}$.
\end{abstract}

\maketitle

\section{Introduction}

Let $\alpha\colon S^1\to\R^{n+1}$ be a closed $C^{n+1}$-differentiable curve such that
for each $t\in S^1$ the vectors
\begin{equation}\label{eq:1}
    \{\alpha'(t), \alpha''(t), \cdots, \alpha^{n}(t)   \}
\end{equation}
are linearly independent. We shall refer to such a curve as a closed \emph{Frenet} curve. The Gram-Schmidt process turns the set of vectors (\ref{eq:1}) into
a set of orthonormal vectors $\{e_1,...,e_n\}$, which, together with a unique unit vector $e_{n+1}$,
forms a positively oriented orthonormal basis $\{e_1,...,e_{n+1}\}$ called the \emph{Frenet} frame~\cite{Kl}.
The \emph{curvatures} $\kappa_1^\alpha, \dots, \kappa_n^\alpha$ of the curve $\alpha$ at the point $t$ are defined by
induction by means of the \emph{Frenet formulae}
\[
   \frac{1}{\left|\alpha'\right|}e'_1=\kappa_1^\alpha e_2,
\]
\[
   \frac{1}{\left|\alpha'\right|}e'_{i+1}= -\kappa_i^\alpha e_i + \kappa_{i+1}^\alpha e_{i+2},  \qquad \textrm{for } i=1,...,n-1,
\]
\[
   \frac{1}{\left|\alpha'\right|}e'_{n+1}=-\kappa_{n}^\alpha e_{n}.
\]
It follows that the curvatures
$\kappa_1^\alpha , ..., \kappa_{n-1}^\alpha$ are strictly positive~\cite[Proposition 1.3.4]{Kl}. We say that the curve $\alpha$ is
\emph{twisted} if the last curvature $\kappa_n^\alpha$ is nowhere zero.

A {\em curvature-like} function is a periodic function to $(\R_{+})^{n-1}\times \R$, that is, a map $S^1\to (\R_{+})^{n-1}\times \R$. 
Every closed Frenet curve $\alpha$ gives rise to the curvature-like function
\[
   j_\alpha \colon t\mapsto ( \kappa_1^\alpha(t), \kappa_2^\alpha(t), ..., \kappa_n^\alpha(t)).
\]
A curvature-like function of the form $j_\alpha$ is called \emph{holonomic}.

Given a curvature-like function $s$ and a real number $\varepsilon>0$, we say that a closed curve $\alpha$
is a \emph{holonomic $\varepsilon$-approximation} of $s$ if $j_\alpha$ is $\varepsilon$-close to $s$, that is,
for each $t\in S^1$ we have $|s(t)-j_\alpha(t)|<\varepsilon$. Similarly, a holonomic approximation of a family $\{s_{u}\}$, where $u$ ranges over an interval $I$, is a family $\{\alpha_u\}_{u\in I}$ of closed curves such that $\alpha_u$ is a holonomic $\varepsilon$-approximation of $s_u$
for each $u\in I$.

\medskip

Our main result is the existence theorem for holonomic approximations.

\begin{theorem}\label{th:1} Every curvature-like function admits a holonomic $\varepsilon$-approximation for every $\varepsilon$. The approximating curve $S^1\to \R^{n+1}$ can be chosen so as to be an embedding.
\end{theorem}

This result is not entirely trivial as its analogues fail both for families of curvature-like functions and for approximations by curvatures of curves in the hyperbolic $(n+1)$-space.

\begin{proposition}\label{p:1} There is a family $\{s_u\}$ of curvature-like functions for $n$ odd that admits no holonomic $\varepsilon$-approximation for some $\varepsilon>0$.
\end{proposition}

Proposition~\ref{p:1} will be deduced below from the Shapiro-Shapiro theorem~\cite{SS}.

For the next proposition note that all the above definitions extend to the case of curves in a general Riemannian manifold (see, for instance, \cite{EC}).

\begin{proposition} There exists a curvature-like function that for sufficiently small $\varepsilon$ admits no holonomic $\varepsilon$-approximation by a closed Frenet curve in $\H^{n+1}$.
\end{proposition}
\begin{proof} It is known that a closed curve in the hyperbolic space of any dimension has curvature $\ge 1$ at some of its points (see, for instance,
\cite[Corollary 1.4]{MC}). Consequently, if a curvature-like function is chosen so that its component corresponding to
$\kappa_1$ is everywhere less than $1/2$, then for $\varepsilon<1/2$ it admits no holonomic
$\varepsilon$-approximation.
\end{proof}

\begin{remark} Our main theorem is obviously related to the Holonomic Approximation Theorem of Eliashberg and Mishachev~\cite{EM}. In fact, a holonomic approximation, that we are interested in, is a solution of an appropriate differential relation \cite{Gr}. We have shown that such a differential relation always has a solution. M.~Ghomi proved~\cite{MG} another version of the Holonomic Approximation Theorem and established the $h$-principle for curves and knots of constant curvature. Our results are different, as well as the proofs.
\end{remark}

\section{Proof of Proposition~\ref{p:1}}

A twisted curve in $\R^{n+1}$ is said to be right-oriented if $\kappa_n>0$. By the Shapiro-Shapiro theorem~\cite{SS},  the number of path components of the space of right-oriented curves in $\R^{n+1}$ for $n$ odd is at least $3$.  Also, Shapiro and Shapiro proved that in this case the space of right-oriented curves is the union of two disjoint open subspaces $\mathbf{ND}$ and $\mathbf{NC}$ such that for some curves $\alpha\in \mathbf{ND}$ and $\beta\in \mathbf{NC}$ the functions $s_0=j_\alpha$ and $s_1=j_\beta$ are homotopic through curvature-like functions $\{s_u\}_{u\in [0,1]}$ with positive last
component of $s_u$ for all $u$.

Assume that for all $\varepsilon>0$, there is a holonomic $\varepsilon$-approximation  $\{\gamma_u\}$ of $\{s_u\}$. Since $\mathbf{ND}$ and $\mathbf{NC}$ are open, for sufficiently small $\varepsilon$, the curve $\gamma_0$ is in $\mathbf{ND}$, and the curve $\gamma_1$ is in $\mathbf{NC}$. Also, for sufficiently small $\varepsilon$,  the curve $\gamma_u$ is right-oriented for all $u$. Thus, for sufficiently small $\varepsilon$, the family $\{\gamma_u\}$ is a homotopy of a curve in $\mathbf{ND}$ to a curve in $\mathbf{NC}$ through right-oriented curves. The existence of such a family contradicts the fact that $\mathbf{ND}$ and $\mathbf{NC}$ are disjoint.

\section{Proof of Theorem~\ref{th:1}}

The second part of the theorem, namely, that the holonomic approximation can be chosen to be an embedding, is a simple observation. For closed curves in $\R^2$ Theorem~\ref{th:1} is a consequence of the converse of the four vertex theorem~\cite{D}, and, therefore, the holonomic approximation can be chosen to be an embedded curve. For curves in $\R^{n+1}$ for $n>1$, we can choose slight, in the $C^{n+1}$ metric, perturbations of the approximating curve to get rid of the self-intersection points. Since $C^{n+1}$-close curves have close curvatures, slight perturbations of a holonomic $\frac{\varepsilon}{2}$-approximation are holonomic $\varepsilon$-approximations.

\medskip

As for the first part, the idea of the proof is as follows.

For a non-closed curve the holonomic approximation problem is trivial: one can always find a curve with any prescribed curvatures.
Using the fact that the curvatures of a curve do not depend on the parametrization, we re-parametrize the curvature-like function so that it is very close to a constant almost everywhere apart from a small subinterval of $S^1$. Then we find separately the desired approximation on this subinterval (by integrating the curvatures) and on its complement (using known formulae for constant curvature curves)  and notice that these approximations can be stitched together.

\medskip

Let us first prove Theorem~\ref{th:1} in the case $n+1=2k$.

\subsection{Piecewise {differentiable} holonomic $\varepsilon$-approximations}

Here we shall prove that for every curvature-like function and every $\varepsilon>0$ there exists a continuous
piecewise  differentiable holonomic $\varepsilon$-approximation
$\alpha$. In fact, the curve $\alpha$ will be constructed as a concatenation of two
 differentiable curves $\beta$ and $\gamma$.

\begin{lemma}\label{l:3}  
Let $k_1, \ldots, k_n$ be real numbers with $k_i>0$ for $i<n$ and $k_n\neq 0$. For each $\varepsilon>0$, there exists $\delta>0$
and $u>0$ with the following properties: given two points $p$ and $q$
such that $|p-q|<\delta$,
there exists a curve $\beta$ from $\beta(0)=p$ to $\beta(u)=q$ such that
\begin{itemize}
\item $|\kappa_i^\beta(t)-k_i|<\varepsilon\ $ for $i=1,...,n$ and all $t\in [0,u]$, and
\item $|\beta^{(i)}(0)-\beta^{(i)}(u)|<\varepsilon\ $ for $i=1,...,n+1$.
\end{itemize}
Furthermore, for any positive orthonormal basis of $\R^{n+1}$ the curve $\beta$ can be chosen so that its Frenet frame
at $p$ coincides with the given orthonormal basis.
\end{lemma}

To prove Lemma~\ref{l:3} we will use curves with
constant curvatures. The Frenet formulae for curves with constant curvature ratios were
integrated by J.~Monterde~\cite{JM}.  All curves in $\R^{2k}$ with constant
non-zero curvatures are of the form
\begin{equation}\label{eq:2}
   \alpha(t) = \sum_{l=1}^{k} A_l\cos(b_lt) + B_l\sin(b_lt),
\end{equation}
where each $b_l$ is the imaginary part of the $l$-th eigenvalue of the coefficient matrix for the system of the Frenet formulae, and
$\{A_1,\dots, A_k, B_1, \dots, B_k\}$ is a basis of $\R^{2k}$ with $|A_l|=|B_l|$ for all $l$. We refer to such curves as {\em helices}.

Note that the helix $\alpha(t)$ lies on a $k$-torus. It is either periodic or its image is dense on a subset of the $k$-torus which is an $s$-torus (see, for instance, \cite{K}). In particular, for any $\delta$ there is an arbitrarily large value of the parameter $u$ such that $\alpha(t)$, together with its first $n+1$ derivatives, is $\delta$-close at $t=u$ to $\alpha(t)$ at $t=0$.

The proof of Lemma~\ref{l:3} is now almost immediate.

\begin{proof}[Proof of Lemma~\ref{l:3}]
We shall look for the curve $\beta(t)$ in the form
\begin{equation}\label{curve}
   \beta(t)=\sum^{k}_{l=1} A_l[\cos(b_lt)+f_l(t)] + B_l[\sin(b_lt)+g_l(t)],
\end{equation}
where $f_l$ and $g_l$ are small functions $[0,u]\to \R$ for each $l=1,...,k$, which vanish in some neigbourhood of zero. Since the curve (\ref{eq:2}) passes arbitrarily close to the point $\alpha(0)$ at arbitrarily large values of $u$, there exists $u$ and a closed curve of the form  (\ref{curve}) which satisfies the two conditions of the lemma.

Consider now the space $R$ of all collections of functions $f_l(t), g_l(t): [0,\infty)\to \R$ with compact support not containing zero. We assume that all the functions have $n+1$ continuous derivatives and take the topology on $R$ as the product topology coming from the $C^{n+1}$-metric. Given $u\in (0,\infty)$, the evaluation map at $u$ is an open map $R\to \R^{2k}$.  The conditions of the lemma, applied to the curve (\ref{curve}) define an open subset  $V\subset R$. In particular, for $u$ such that there is a curve $\beta$ with $\beta(0)=\beta(u)=p$, the image of $V$ contains a neighbourhood of $p$ in $\R^{2k}$.

To prove the last statement of Lemma~\ref{l:3} note that the composition of a Frenet curve with a rotation in $\R^{2k}$ is also a Frenet curve. 

\end{proof}

The fact that the curvatures of a curve do not depend on the parametrization is reflected in the following lemma.
\begin{lemma}\label{l:1} Let $q\colon S^1\to S^1$ be a diffeomorphism, and $s$ a curvature-like function.
If $\alpha(t)$ is a holonomic $\varepsilon$-approximation of the function $s\circ q(t)$, then $\alpha\circ q^{-1}(t)$ is a holonomic $\varepsilon$-approximation of $s$.
\end{lemma}

Let $s$ be a curvature-like function and $\varepsilon>0$ a real number. Without loss of generality we can 
assume that there exists a very small neighbourhood $U$ on which $s$ is constant and where $s(t)=(k_1,\ldots,k_n)$ with $k_n \neq 0$.
 
\begin{figure}[ht]
	\centering
			\includegraphics[draft=false, width=100mm]{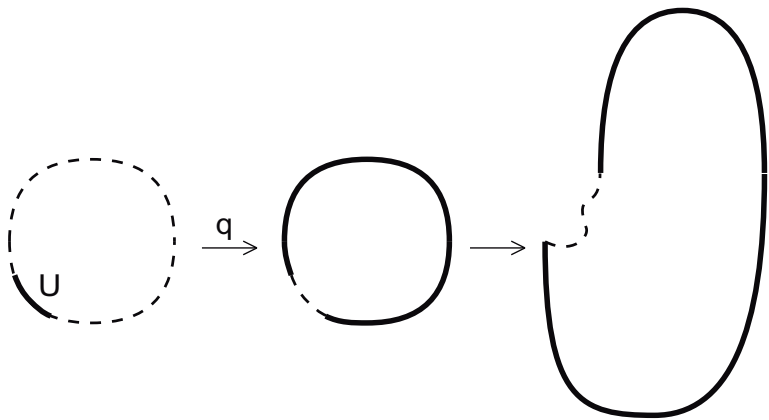}
\caption{The map $q$ and the concatenation of $\beta$ and $\gamma$.}
\label{fig:1}
\end{figure}

On the other hand, according to Lemma~\ref{l:1}, we do not need to think that $U$ is very small. In fact, for any $\delta>0$ we can choose a re-parametrization $q:S^1\to S^1$ so that if $\gamma$ is a curve $$(S^1-U)\to\R^{2k}$$ parametrized by arc length and such that $j_\gamma = s$ on $S^1-U$, then the length of $\gamma$ is smaller than $\delta$. In other words, we choose a re-parametrization of $S^1$ for which the length of $S^1-U$ is smaller than $\delta$.

In particular, for a given $\varepsilon$ we can choose $\delta>0$ as in Lemma~\ref{l:3} for the helix of the curvatures  $k_1,\ldots,k_n$. By Lemma~\ref{l:3} the two ends of $\gamma$ can be joined by a curve $\beta$ which, together with all its curvatures and derivatives of order $\le n+1$, is $\varepsilon$-close to such a helix. 
Adjust the parameter of this curve $\beta$, without affecting the derivatives near the ends,  so that it varies over the closure of $U$; then by construction we have that $j_{\beta}$ is $\varepsilon$-close to $s$ over $U$. 

The closed curve that coincides with $\gamma$ over $S^1-U$ and with $\beta$ over $U$ is a holonomic $\varepsilon$-approximation to $s$ everywhere except for the
end points of $\gamma$ and $\beta$ where it may not be differentiable.

\subsection{Improving piecewise differentiable holonomic approximations}

To begin with we observe that by Lemma~\ref{l:3} we may assume that the Frenet frames of $\gamma$ and $\beta$ coincide at one of the end points. Let $\tau\in S^1$ be the other endpoint of $\gamma$ and $\beta$.

\begin{lemma}  For any prescribed $\varepsilon'$ the number $\delta$ can be chosen so that the Frenet
frames of the curves $\gamma$ and $\beta$ are $\varepsilon'$-close at $\tau$.
\end{lemma}
\begin{proof}
If $\delta$ is sufficiently small, then by the last inequality in Lemma~\ref{l:3} the Frenet frames of $\beta$ at the end points are
$\varepsilon'/2$-close. On the other hand, if $\delta$ is
sufficiently small, then the Frenet frames of $\gamma$ at the initial and end
moments $t_1$ and $t_2$ are also $\varepsilon'/2$-close. Indeed,
\[
    e_1(t_2)=e_1(t_1) + \int_{t_1}^{t_2} \kappa_1^{\gamma}(t)e_2(t)dt \approx e_1(t_1),
\]
\[
    e_{i+1}(t_2)= e_{i+1}(t_1) + \int_{t_1}^{t_2} [\kappa_{i+1}^\gamma e_{i+2}
    - \kappa_i^\gamma e_i(t)]dt \approx e_{i+1}(t_1),
\]
for $i=1,...,n-1$, and
\[
    e_{n+1}(t_2)= e_{n+1}(t_1) - \int_{t_1}^{t_2} \kappa_{n}^\gamma e_{n}(t)dt \approx e_{n+1}(t_1).
\]

\end{proof}

Finally,  it is clear from the construction that we may assume that the curvatures $\kappa^{\gamma}_i$ are locally constant at $\tau$ and that the same is true for $\kappa^{\beta}_i$.  This condition, while not strictly necessary,  allows us to simplify the statement of Lemma~\ref{next} below.

\medskip

It only remains to prove that there exists an approximation of the concatenation of $\beta$ and $\gamma$ by a differentiable curve whose curvatures
is close to the curvatures for the concatenation of $\beta$ and $\gamma$ at all points of $S^1-\{\tau\}$.
Such curve can be obtained by using well-known tools such as the the standard mollifiers (see \cite[Appendix C.4]{E}). In particular, we have the following statement, which finishes our proof:
\begin{lemma}\label{next}
Let $\alpha:\R\to\R^{n+1}$ be a continuous curve which for $t>0$ and for $t<0$ coincides with helices of the curvatures $(k_1,\ldots,k_n)$. For each $\varepsilon>0$ there exists $\varepsilon'>0$ such that if the Frenet frames of $\alpha$ at $t\to-0$ and  at $t\to+0$ are $\varepsilon'$-close, then there exists a $C^{\infty}$-curve $\tilde{\alpha}$ 
with $$|\tilde{\alpha}-\alpha|_{C^{n+1}}<\varepsilon$$ in an arbitrarily small neighbourhood of 0 and  $\tilde{\alpha}=\alpha$ outside this neighbourhood.
\end{lemma}

\subsection{The case $n=2k$}

As shown in \cite{JM}, all curves in $\R^{2k+1}$ with constant curvatures are of the form 
\[
   \alpha(t) = A_0t+\sum_{l=1}^{k} A_l\cos(b_lt) + B_l\sin(b_lt), 
\]
where each $b_l$ is the imaginary part of the $l$-th eigenvalue of the coefficient matrix of the system of the
Frenet formulae, and
$\{A_0,\dots, A_k, B_1, \dots, B_k\}$ is a basis of $\R^{2k+1}$ with $|A_l|=|B_l|$ for all $l$. Unlike the helices in $\R^{2k}$, these curves do not 
lie on tori but on cylinders on which they are not dense. Nevertheless, Lemma~\ref{l:3} is still true and the proof is the same as that for curves in $\R^{2k}$ except 
that we find an approximation in a neighbourhood of the curve
\[
   A_0r\cos(t) + A_1r\sin(t) +\sum_{l=1}^{k} A_l\cos(b_lt) + B_l\sin(b_lt), 
\]
where $r$ is a sufficiently big real number. The rest of the proof of Theorem~\ref{th:1} is also virtually unchanged as compared to the case of curves in $\R^{2k}$.

\end{document}